\documentclass[reqno,12pt,centertags]{amsart}
\usepackage{amsmath,amsthm,amscd,amssymb,latexsym,upref,stmaryrd,epsfig,graphicx}
\usepackage[numbers,sort&compress]{natbib}

\usepackage{hyperref}
\newcommand*{\mailto}[1]{\href{mailto:#1}{\nolinkurl{#1}}}

\newcommand{\beq}{\begin{equation}}
\newcommand{\eeq}{\end{equation}}
\newcommand{\ba}{\begin{align}}
\newcommand{\ea}{\end{align}}

 \allowdisplaybreaks
\numberwithin{equation}{section}
\newtheorem{theorem}{Theorem}[section]
\theoremstyle{definition}

\newtheorem{remark}{Remark}[section]

\begin{document}


\title[Discontinuous Hochstadt-Lieberman Problems]
{Reconstruction and Solvability for Discontinuous Hochstadt-Lieberman Problems}
\date{\today}

\author[C.~F.~Yang]{Chuan-Fu Yang}

\address{Department of Applied Mathematics, School of Science, Nanjing University of
Science and Technology, Nanjing 210094, Jiangsu, People's Republic
of China}

\email{chuanfuyang@njust.edu.cn}

\author[N.~Bondarenko]{Natalia Bondarenko}

\address{1. Department of Applied Mathematics and Physics, Samara National Research University, Samara, Russia
\ \ 2. Saratov State University, Saratov, Russia}

\email{bondarenkonp@info.sgu.ru}



\maketitle

\begin{abstract}
  We consider Sturm-Liouville problems with a discontinuity in an interior point, which are motivated by the inverse problems for the torsional modes of the Earth.
  We assume that the potential on the right half-interval and the coefficient in the right boundary condition are given.
  Half-inverse problems are studied, that consist in recovering the potential on the left half-interval and the left boundary condition from the eigenvalues.
  If the discontinuity belongs to the left half-interval, the position and the parameters of the discontinuity also can be reconstructed.
  In this paper, we provide reconstructing algorithms and prove existence of solutions for the considered inverse problems.
  Our approach is based on interpolation of entire functions.
\end{abstract}
\maketitle

\section{Introduction}

The paper concerns the theory of inverse spectral problems for differential operators. Such problems consist in recovering operators
from their spectral characteristics.

This paper is focused on the eigenvalue problem
\begin{equation}\label{SL}
  -y''+q(x)y=\lambda^2 y, \ 0<x<1,
\end{equation}
with the boundary conditions
\begin{equation}\label{bcs}
  y'(0)-h_1y(0)=0=y'(1)+h_2y(1)
\end{equation}
and with the jump conditions
\begin{equation}\label{jcs}
  y(d+)=a_1y(d-), \ y'(d+)=a^{-1}_{1}y'(d-)+a_2y(d-),
\end{equation}
where $q\in L^2(0,1)$, $0<d\leq 1/2, a_1>0$ and $|a_1-1|+|a_2|>0$.

This problem appears in the inverse problems for the torsional modes of the Earth. Here the discontinuity is mainly caused by reflection of the shear waves at the base of the
crust.

The most complete results in inverse problem theory are obtained for the Sturm-Liouville equation~\eqref{SL} without discontinuities
(see the monographs \cite{FY, LE, MA, PT}). In particular, Borg \cite{BO, FY} has proved, that the Sturm-Liouville potential $q(x)$ is uniquely determined
by two spectra, corresponding to different boundary conditions. However, in some special cases, only one spectrum is sufficient.
The classical Hochstadt-Lieberman theorem \cite{HL} states that, in the case when $a_1=1$ and $a_2=0,$ if
the parameters $h_1$ and $h_2$ of the boundary conditions for the problem (\ref{SL})-(\ref{jcs}) are fixed and the potential $q$ is given on the right half-interval,
then the potential is uniquely determined by the eigenvalues $\{\lambda_n^2\}_{n\geq 0}$.
In fact, the restriction of left boundary parameter $h_1$ is unnecessary (see Hald \cite{HA1}).

Later Hald generalized the theorem by Hochstadt and Lieberman to the Sturm-Liouville operator with the discontinuity \cite{HA}.
Hald supposed, that the discontinuity is in the left half-interval,
the potential is known over the right-half interval, and the right boundary condition is given.
Then the potential and the left boundary condition are uniquely determined by the eigenvalues.
The position of the discontinuity and the jump in the eigenfunctions are also uniquely determined.
In particular, those Hald's results are valid, when the discontinuity is in the middle of the interval. However, a counter-example shows that,
by using only the eigenvalues, one cannot in general determine the jump in the eigenfunctions (see p.560 in \cite{HA}).

Let us formulate the Hald's uniqueness theorems, generalizing the theorem by Hochstadt and Lieberman:

\medskip

\textbf{Theorem A  (Uniqueness).} \emph{Let $d=1/2$ and $\{\lambda^2_n\}_{n\geq 0}$ be the eigenvalues of the problem
(\ref{SL})-(\ref{jcs}). Then the data $\left\{\{\lambda^2_n\}_{n\geq 0},a_1, h_2,q|_{[1/2,1]}\right\}$
uniquely determine $q$ almost everywhere on the interval $[0,1]$, the quantities $a_2$ and $h_1$.}\\

\textbf{Theorem B (Uniqueness).} \emph{Let $0<d<1/2$ and $\{\lambda^2_n\}_{n\geq 0}$ be the eigenvalues of the problem
(\ref{SL})-(\ref{jcs}). Then the data $\left\{\{\lambda^2_n\}_{n\geq 0},h_2,q|_{[1/2,1]}\right\}$
uniquely determine $q$ almost everywhere on the interval $[0,1]$ and the quantities $\{a_1,a_2,d,h_1\}$.}\\

Note that for the inverse problems of spectral analysis, the following most important issues are usually studied:

\begin{enumerate}
\item uniqueness,

\item constructive methods for solution,

\item necessary and sufficient conditions of solvability.
\end{enumerate}

Uniqueness theorems for the classical Hochstadt-Lieberman problems without discontinuities
were obtained, e.g., in \cite{HL, GS}. We mention that these problems are also called the half-inverse problems and the inverse problems by mixed data.
Later on, constructive methods for solving the Hochstadt-Lieberman problems have been developed
and conditions for existence of solutions have been obtained by Sakhnovich \cite{SA}, Martinyuk and Pivovarchik \cite{MP, PI},
Hryniv and Mykytyuk \cite{HM}. Reconstruction procedures for solving half-inverse problems for
Sturm-Liouville operators and pencils were also developed by Buterin \cite{B1,B2}. However, for the inverse problems for
the Sturm-Liouville operator with the discontinuity (\ref{SL})-(\ref{jcs}), only uniqueness theorems have been proved before
(see \cite{HA1, HA, SBI}).
The goals of this paper are to develop methods for constructing solutions of these problems and to present
conditions for existence of solutions.

In this paper, we focus on the inverse problems, corresponding to Theorems~A and~B. In order to solve these inverse problems constructively,
we develop the ideas of \cite{PI}. Our technique is based on the representation of the characteristic function for the whole interval,
by using the characteristic functions for the half-intervals. The main tool of our method is interpolation of entire functions (see \cite{PI, B1, B2}).
However, we use slightly different approaches for $d = 1/2$ and $d < 1/2$. In the case $d = 1/2$,
the parameters of the jump conditions~\eqref{jcs} are supposed to be known,
and our half-inverse problem is reduced to the complete inverse problem on the interval $(0, 1/2)$.
The resulting problem does not have a discontinuity, so we can solve it by standard methods of inverse problem theory.
In particular, the transformation operator method by Marchenko \cite{MA} is applied. In the case $0 < d < 1/2$, we first reconstruct the parameters
$d$ and $a_1$ of the jump condition. Then we pass to the inverse problem for the Sturm-Liouville equation with a discontinuity on the left half-interval.
For solving this problem, the method of spectral mappings \cite{FY} is applied.
For the both half-inverse problems, corresponding to Theorems~A and~B,
we obtain conditions, necessary and sufficient for existence of solution. They are formulated in Theorems~2.1 and~3.1.

\section{The case $d=1/2$}
This section deals with the existence result, corresponding to Theorem A, and the reconstructing algorithm for solving the corresponding inverse problem.

\subsection{Direct problems}

In this subsection, some preliminaries are provided. In particular, we represent the characteristic function of the problem (\ref{SL})-(\ref{jcs}),
using the characteristic functions for the half-intervals (see the relation~\eqref{2.7} below). We also describe the properties of the eigenvalues, and construct the characteristic function
by its zeros as an infinite product.	

Let $d=1/2$. Consider (\ref{SL})-(\ref{jcs}) as a boundary value problem on a star graph, and rewrite it as follows:
\begin{align}
   &-y''_j+q_j(x)y_j=\lambda^2y_j, \quad x\in\left(0,1/2\right),j=1,2,  \label{2.1}\\
   &y'_j(0)-h_jy_j(0)=0, \label{2.2} \\
   &y_2\left(1/2\right)=a_1y_1\left(1/2\right),  \label{2.3}\\
   &y'_2\left(1/2\right)+a^{-1}_1y'_1\left(1/2\right)+a_2y_1\left(1/2\right)=0.\label{2.4}
\end{align}
Here $q_j\in L^2(0,1/2)$ for $j=1,2$, and $q_1(x):=q(x)|_{[0,1/2]}$, $q_2(x):=q(1-x)|_{[0,1/2]}$.

Denote by $\varphi_j(x,\lambda)$ the solution of Eq.(\ref{2.1}) with $\varphi_j(0,\lambda)-1=\varphi'_j(0,\lambda)-h_j=0$. Then, according to \cite{FY,MA},
the following relations hold
\begin{equation}\label{2.5}
  \varphi_j\left(1/2,\lambda\right)=\cos\frac{\lambda}{2}+\left(h_j+[q_j]\right)\frac{\sin\frac{\lambda}{2}}{\lambda}+\frac{\psi^{(j)}_1(\lambda)}{\lambda},
\end{equation}
\begin{equation}\label{2.6}
  \varphi'_j\left(1/2,\lambda\right)=-\lambda\sin\frac{\lambda}{2}+\left(h_j+[q_j]\right)\cos\frac{\lambda}{2}+\psi^{(j)}_2(\lambda),
\end{equation}
where $\psi^{(j)}_1,\psi^{(j)}_2\in \mathcal{L}^{1/2}$ ($\mathcal{L}^{a}$ is the class of entire functions of exponential type $\leq a$, belonging to $L^2(\mathbb{R})$ for real
$\lambda$), and for an integrable function $f$ in $[0,1/2]$, we define
$$
[f]:=\frac{1}{2}\int_0^{1/2}f(x)dx.
$$

The solution of the problem (\ref{2.1})-(\ref{2.2}) possesses the form $y_j=c_j\varphi_j(x,\lambda)$
$(j=1,2)$, where $c_j$ are constants. Taking (\ref{2.3}) and (\ref{2.4}) into account, we represent
the characteristic function of the problem (\ref{2.1})-(\ref{2.4}) in the form
\begin{equation}\label{2.7}
\begin{split}
\Phi(\lambda)=&a_1\varphi_1\left(1/2,\lambda\right)\varphi'_2\left(1/2,\lambda\right)
+a^{-1}_1\varphi'_1\left(1/2,\lambda\right)\varphi_2\left(1/2,\lambda\right)\\
&+
a_2\varphi_1\left(1/2,\lambda\right)\varphi_2\left(1/2,\lambda\right).
\end{split}
\end{equation}
Using (\ref{2.5}) and (\ref{2.6}), we see that \eqref{2.7} is equivalent to
\begin{equation}\label{2.8}
   \begin{split}
   \Phi(\lambda)\!=&-\frac{a_1+a_1^{-1}}{2}\lambda\sin\lambda+1/2(a_2+(a_1-a^{-1}_1)\left([q_2]+h_2\right.\\
        &\left.\!\!-[q_1]-h_1\right))\!\!+\!\!(a_2+(a_1+a^{-1}_1)\left([q_1+q_2]\right.\\
        &+h_1+h_2))\frac{\cos\lambda}{2}+\psi_3(\lambda),
   \end{split}
\end{equation}
where $\psi_3\in \mathcal{L}^1$.

The eigenvalues $\{ \lambda_n^2 \}_{n \ge 0}$ of the boundary value problem (\ref{SL})-(\ref{jcs}) are real and simple.
We suppose that they are numbered in such a way
that the following asymptotic formula holds (see \cite{YA}):
\begin{equation} \label{asymptla}
  \lambda_n=n\pi+\frac{(-1)^na+b}{n\pi}+\frac{\beta_n}{n},
\end{equation}
where
\begin{equation} \label{defa}
  a=\frac{a_2}{a_1+a_1^{-1}}+\frac{a_1-a_1^{-1}}{a_1+a_1^{-1}}\left([q_2]+h_2-
  [q_1]-h_1\right),
\end{equation}
\begin{equation} \label{defb}
  b=\frac{a_2}{a_1+a_1^{-1}}+[q_1+q_2]+h_1+h_2.
\end{equation}
Here and below the notation $\{\beta_n\}^{\infty}_{n=0}$
is used for different sequences belonging to $l^2$.

By using the numbers $\{\lambda_n\}_{n\ge 0}$, we construct
\begin{equation}\label{s1}
  \Phi(\lambda)=C\prod_{n=0}^{\infty}\left(1-\frac{\lambda^2}{\lambda_n^2}\right),
\end{equation}
where $C > 0$ is a constant.
One may suppose that $\lambda_n\neq 0$, otherwise these numbers may be shifted by
a constant.
In view of the representation (\ref{2.8}), we obtain the formula for the constant $C$:
\begin{equation*}
  C=-\frac{a_1+a^{-1}_1}{2}\lim_{m\rightarrow\infty}\left[\frac{1}{\frac{\pi}{2}+2m\pi}\prod_{n=0}^{\infty}\left(1-\frac{(\frac{\pi}{2}+2m\pi)^2}{\lambda_n^2}\right)\right]^{-1}.
\end{equation*}

We also note that, by using $\{ \lambda_n^2 \}_{n \ge 0}$, $q_2$ and $h_2$, we can construct the values $a_2$ and $[q_1] + h_1$.
Indeed, introduce the numbers
\begin{equation} \label{defga}
\gamma_n := (\lambda_n - n \pi) n \pi, \quad n \ge 0.
\end{equation}
In view of the asymptotic relation~\eqref{asymptla}, the numbers $a$ and $b$ can be calculated by the following formulas:
\begin{equation} \label{findab}
a = \frac{1}{2} \lim_{n \to \infty} (\gamma_{2n} - \gamma_{2n+1}),
\quad b = \frac{1}{2} \lim_{n \to \infty} (\gamma_{2n} + \gamma_{2n + 1}).
\end{equation}
Solving the system of linear equations~\eqref{defa} and~\eqref{defb}, one can find
\begin{gather} \label{om1}
    [q_1] + h_1 = -\frac{1}{2 a_1} ((a_1 + a_1^{-1}) (a - b) + 2 a_1^{-1} (h_2 + [q_2])), \\ \label{defa2}
    a_2 = (b - [q_1] - h_1 - [q_2] - h_2)(a_1 + a_1^{-1}).
\end{gather}

\subsection{Reconstruction of $\varphi_1(1/2,\lambda)$ and $\varphi'_1(1/2,\lambda)$}

Suppose that the eigenvalues $\{\lambda_n^2\}_{n\ge 0}$, the potential $q_2$, the coefficient $h_2$ of the boundary condition and the coefficient $a_1$
of the jump conditions are given. By the formulas~\eqref{om1} and~\eqref{defa2} we can calculate $[q_1] + h_1$ and $a_2$.
Our goal is to find $q_1$ and $h_1$. In this section, we focus on the auxiliary step of recovering
the characteristic functions $\varphi_1(1/2,\lambda)$ and $\varphi'_1(1/2,\lambda)$, associated with the left half-interval.
Our technique is based on interpolation of entire functions (see the relations~\eqref{2.17} and~\eqref{2.24} below).

Note that the potential $q_2$ and the coefficient $h_2$ are given, so we can find the function $\varphi_2(1/2,\lambda)$ and its zeros $\{\nu_n^{(1)}\}_{n\in \mathbb{Z}^0}$
as well as the function $\varphi'_2(1/2,\lambda)$ and its zeros $\{\mu_n^{(1)}\}_{n\in \mathbb{Z}^0}$,
$\mathbb Z^0 := \{ \pm 0, \pm 1, \pm 2, \ldots \}$. The following asymptotic relations are valid for $n \ge 0$ (see \cite{YA}):
\begin{equation}\label{2.9}
  \nu_n^{(1)}=(2n+1)\pi+\frac{2[q_2]+2h_2}{(2n+1)\pi}+\frac{\beta_n}{n},
\end{equation}
\begin{equation}\label{2.10}
  \mu_n^{(1)}=2n\pi+\frac{[q_2]+h_2}{n\pi}+\frac{\beta_n}{n},
\end{equation}
and $\nu_{-n}^{(1)} = -\nu_n^{(1)}$, $\mu_{-n}^{(1)} = -\mu_n^{(1)}$.

Letting $\lambda=\nu_n^{(1)}$ in Eq.(\ref{2.7}), we have
\begin{equation}\label{2.11}
  \varphi_1\left(1/2,\nu_n^{(1)}\right)=\frac{\Phi(\nu_n^{(1)})}{a_1\varphi'_2(1/2,\nu_n^{(1)})}.
\end{equation}

We know from (\ref{2.5}), that for determination of $\varphi_1(1/2,\lambda)$, it is sufficient to recover $\psi_1^{(1)}(\lambda)$. By choosing $\{\nu_n^{(1)}\}_{n\in \mathbb{Z}^0}$
as the nodes of interpolation, we find the function $\psi_1^{(1)}(\lambda)$. In order to do this, we first calculate the values of the
function $\psi_1^{(1)}(\lambda)$ at the nodes.
The relations (\ref{2.5}), (\ref{2.6}), (\ref{2.9}) and (\ref{2.11}) yield
\begin{equation}\label{2.12}
    \begin{split}
      \psi_1^{(1)}(\nu_n^{(1)})=\nu_n^{(1)}\bigg[\frac{\Phi(\nu_n^{(1)})}{a_1\varphi'_2(1/2,\nu_n^{(1)})}-\cos\frac{\nu_n^{(1)}}{2} -\left(h_1+[q_1]\right)\frac{\sin\frac{\nu_n^{(1)}}{2}}{\nu_n^{(1)}}\bigg].
    \end{split}
\end{equation}

Consider the indices $n \ge 0$.
In order to estimate the values $\psi_1^{(1)}(\nu_n^{(1)})$, we need the following relations:
\begin{equation*}
  \begin{split}
     \Phi(\nu_n^{(1)})=&\frac{a_1+a_1^{-1}}{2}\left(2[q_2]+2h_2\right)-a_1[q_1] \\
                       &-a_1^{-1}[q_2]-a_1h_1-a_1^{-1}h_2+\beta_n,
   \end{split}
\end{equation*}
\begin{equation*}
  a_1\varphi'_2(1/2,\nu_n^{(1)})=(-1)^{n-1}a_1(2n+1)\pi(1+\beta_n).
\end{equation*}
Hence
\begin{equation}\label{2.13}
  \frac{\Phi(\nu_n^{(1)})}{a_1\varphi'_2(1/2,\nu_n^{(1)})}=\frac{[q_2]+h_2-[q_1]-h_1}{(-1)^{n-1}(2n+1)\pi}(1+\beta_n).
\end{equation}
Note that
\begin{equation}\label{2.14}
  \cos\frac{\nu_n^{(1)}}{2}=(-1)^{n-1}\frac{[q_2]+h_2}{(2n+1)\pi}+\frac{\beta_n}{n}
\end{equation}
and
\begin{equation}\label{2.15}
  \frac{\sin\frac{\nu_n^{(1)}}{2}}{\nu_n^{(1)}}=\frac{(-1)^{n}}{(2n+1)\pi}(1+\beta_n).
\end{equation}
Substituting (\ref{2.13})-(\ref{2.15}) into (\ref{2.12}) yields
\begin{equation*}
  \begin{split}
    \psi_1^{(1)}(\nu_n^{(1)})=&\nu_n^{(1)}\bigg[\frac{[q_2]+h_2-[q_1]-h_1}{(-1)^{n-1}(2n+1)\pi}  \\
                              &-\frac{h_2+[q_2]}{(-1)^{n-1}(2n+1)\pi}-\frac{h_1+[q_1]}{(-1)^{n}(2n+1)\pi}+\frac{\beta_n}{n}\bigg]\\
                              =&\beta_n, \quad n \ge 0.
  \end{split}
\end{equation*}

Since $\psi_1^{(1)}(\nu_n^{(1)}) = -\psi_1^{(1)}(\nu_{-n}^{(1)})$, we get
\begin{equation}\label{2.16}
  \{\psi_1^{(1)}(\nu_n^{(1)})\}_{n\in \mathbb{Z}^0}\in l^2.
\end{equation}

Since the function $\varphi_2(1/2,\lambda)$ is of sine-type (see \cite{LL}), together with (\ref{2.16}), we can apply Theorem A in \cite{LL}
and obtain
\begin{equation}\label{2.17}
  \psi_1^{(1)}(\lambda)=\varphi_2\left(1/2,\lambda\right)\sum_{n\in \mathbb{Z}^0}\frac{\psi_1^{(1)}(\nu_n^{(1)})}{\frac{d\varphi_2(1/2,\lambda)}{d\lambda}|_{\lambda=\nu_n^{(1)}}(\lambda-\nu_n^{(1)})}.
\end{equation}
The series on the right-hand side of (\ref{2.17}) converges uniformly on any compact subdomain of $\mathbb{C}$ and in $L^2(\mathbb{R})$ for real $\lambda$
to a function, which belongs to
$\mathcal{L}^{1/2}$. Substituting (\ref{2.17}) into (\ref{2.5}), one can construct $\varphi_1(1/2,\lambda)$.

Next we begin to find $\varphi'_1(1/2,\lambda)$. Replacing $\lambda$ by $\mu_n^{(1)}$ in Eq.(\ref{2.7}), we get
\begin{equation}\label{2.18}
  \varphi'_1\left(1/2,\mu_n^{(1)}\right)=\frac{\Phi(\mu_n^{(1)})-a_2\varphi_1(1/2,\mu_n^{(1)})\varphi_2(1/2,\mu_n^{(1)})}{a_1^{-1}\varphi_2(1/2,\mu_n^{(1)})}.
\end{equation}
Consider the indices $n \ge 0$. It follows from (\ref{2.5}), (\ref{2.6}), (\ref{2.10}) and (\ref{2.18}), that
\begin{equation}\label{2.19}
  \begin{split}
    \psi_2^{(1)}(\mu_n^{(1)})=&\varphi'_1\left(1/2,\mu_n^{(1)}\right)+\mu_n^{(1)}\sin\frac{\mu_n^{(1)}}{2}-\left(h_1+[q_1]\right)\cos\frac{\mu_n^{(1)}}{2} \\
                             =&\frac{\Phi(\mu_n^{(1)})}{a_1^{-1}\varphi_2(1/2,\mu_n^{(1)})}-a_1a_2
                             \varphi_1\left(1/2,\mu_n^{(1)}\right)+\mu_n^{(1)}\sin\frac{\mu_n^{(1)}}{2}\\
                             &-\left(h_1+[q_1]\right)\cos\frac{\mu_n^{(1)}}{2}.
  \end{split}
\end{equation}
In order to estimate $\psi_2^{(1)}(\mu_n^{(1)})$, we need the following relations:
\begin{equation*}
  \begin{split}
     \Phi(\mu_n^{(1)})=&-\frac{a_1+a_1^{-1}}{2}\left(2[q_2]+2h_2\right)+a_1[q_2] \\
                       &+a_1^{-1}[q_1]+a_1h_2+a_1^{-1}h_1+a_2+\beta_n,
   \end{split}
\end{equation*}
\begin{equation*}
  a_1^{-1}\varphi_2\left(1/2,\mu_n^{(1)}\right)=(-1)^{n}a_1^{-1}+\beta_n.
\end{equation*}
The latter formulas yield
\begin{equation}\label{2.20}
  \begin{split}
    \frac{\Phi(\mu_n^{(1)})}{a_1^{-1}\varphi_2(1/2,\mu_n^{(1)})}=&(-1)^{n-1}\frac{a_1+a_1^{-1}}{2a_1^{-1}}\left(2[q_2]+2h_2\right) \\
                                                                      &+(-1)^n a_1^2[q_2]+(-1)^n[q_1]\\
                                                                      &+(-1)^n(a_1^2h_2+h_1+a_1a_2)+\beta_n.
  \end{split}
\end{equation}

We also have
\begin{equation}\label{2.21}
  a_1a_2\varphi_1\left(1/2,\mu_n^{(1)}\right)=(-1)^na_1a_2+\beta_n,
\end{equation}
\begin{equation}\label{2.22}
  \mu_n^{(1)}\sin\frac{\mu_n^{(1)}}{2}=(-1)^n\left([q_2]+h_2\right)+\beta_n,
\end{equation}
\begin{equation}\label{2.23}
  \left(h_1+[q_1]\right)\cos\frac{\mu_n^{(1)}}{2}=(-1)^n\left(h_1+[q_1]\right)+\beta_n.
\end{equation}
Substituting (\ref{2.20})-(\ref{2.23}) into (\ref{2.19}), and taking the equality
$\psi_2^{(1)}(\mu_n^{(1)}) = \psi_2^{(1)}(\mu_{-n}^{(1)})$ into account,  we get
\begin{equation*}
\{\psi_2^{(1)}(\mu_n^{(1)})\}_{n \in \mathbb Z^0} \in l^2.
\end{equation*}
Taking into account, that the function $g_1(\lambda):=\frac{\lambda\varphi'_2(1/2,\lambda)}{\lambda^2-(\mu_0^{(1)})^2}$ is of sine type, we again use interpolation (see \cite{LL}):
\begin{equation}\label{2.24}
  \psi_2^{(1)}(\lambda)=g_1(\lambda)\sum_{0\neq n=-\infty}^{+\infty}\frac{\psi_2^{(1)}(\mu_n^{(1)})}{\frac{dg_1(\lambda)}{d\lambda}|_{\lambda=\mu_n^{(1)}}(\lambda-\mu_n^{(1)})}+g_1(\lambda)\frac{\psi_2^{(1)}(0)}{g_1'(0)\lambda}.
\end{equation}
Substituting (\ref{2.24}) into (\ref{2.6}), we can construct $\varphi'_1(1/2,\lambda)$.

Here
\begin{equation}\label{ew}
  \begin{split}
    &\psi_2^{(1)}(0) \\
    &=\varphi_1'(1/2,0)-(h_1+[q_1]) \\
   &=\frac{\Phi(0)-a_1\varphi_1(1/2,0)\varphi_2'(1/2,0)-a_2\varphi_1(1/2,0)\varphi_2(1/2,0)}{a_1^{-1}\varphi_2(1/2,0)}-(h_1+[q_1]).
  \end{split}
\end{equation}
One can deal with the case $\varphi_2(1/2,0)=0$, shifting the spectrum by a constant.

\subsection{Inverse problems: existence and algorithm}

In this subsection, we prove the first of our main results. The following theorem establishes the existence of solution for
the considered half-inverse problem.

\begin{theorem}\label{1}
Let data $S:=\left\{a_1, h_2, q_2(x), \{\lambda_n^2\}_{n \ge 0}\right\}$
satisfy the following conditions:
\begin{enumerate}
   \item $\lambda_0^2<\lambda_1^2<\dots<\lambda_n^2<\dots$;
   \item $\lambda_n=n\pi+\frac{(-1)^na+b}{n\pi}+\frac{\beta_n}{n}$, $\{\beta_n\}_{n\ge 0}\in l^2$, $a,b\in\mathbb{R}$ and,
         if $a_1 = 1$, then $a \ne 0$;
   \item The function $M(\mu) := \frac{\varphi_1(1/2,\lambda)}{\varphi'_1(1/2,\lambda)}$, $\mu = \lambda^2$,
            is a Nevanlinna function of $\mu$, where $\varphi_1(1/2, \lambda)$ and $\varphi'(1/2, \lambda)$
            are the functions, constructed in Section~2.2.
\end{enumerate}
Then there exist a real-valued function $q_1(\cdot)\in L^2(0,1/2)$ and real numbers $h_1$, $a_2$, such that the spectrum
of the problem (\ref{2.1})-(\ref{2.4}), generated by $(a_1, a_2, h_1, h_2, q_1, q_2)$, coincides with $\{\lambda_n^2\}_{n\geq 0}$.
\end{theorem}

\begin{proof}
Let the data $S$ fulfill the conditions (1)-(3), and $\varphi_1(1/2,\lambda)$, $\varphi_1'(1/2, \lambda)$
be the functions, constructed in Section~2.2.
Since the function $M(\mu) = \frac{\varphi_1(1/2,\lambda)}{\varphi'_1(1/2,\lambda)}$ belongs to the Nevanlinna class,
then its zeros $\{(\nu_n^{(0)})^2\}_{n\geq 0}$ interlace with its poles $\{(\mu_n^{(0)})^2\}_{n\geq 0}$:
\begin{equation*}
  (\mu_0^{(0)})^2<(\nu_0^{(0)})^2<\cdots<(\mu_n^{(0)})^2<(\nu_n^{(0)})^2<\cdots.
\end{equation*}
The expressions~(\ref{2.5}) and (\ref{2.6}) imply that
\begin{equation*}
  \nu_n^{(0)}=(2n+1)\pi+\frac{\kappa}{(2n+1)\pi}+\frac{\beta_n}{n},
\end{equation*}
\begin{equation*}
  \mu_n^{(0)}=2n\pi+\frac{\kappa}{2n\pi}+\frac{\beta_n}{n},
\end{equation*}
where
\begin{equation*}
  \begin{split}
  \kappa=2b-\frac{2a_2}{a_1+a_1^{-1}}-2h_2-2[q_2].
  \end{split}
\end{equation*}
Here $a_1, h_2$ and $[q_2]$ are known, and $a_2$ is constructed by~\eqref{defa2}.
Thus, the sets $\{(\nu_n^{(0)})^2\}_{n\geq 0}$ and $\{(\mu_n^{(0)})^2\}_{n\geq 0}$ satisfy the conditions of Theorem 3.4.1 in \cite{MA},
so there exists a unique real-valued function $q_1(\cdot)\in L^2(0,1/2)$, that generates Robin-Dirichlet and Robin-Neumann problems
on $[0,1/2]$ with the spectra $\{(\nu_n^{(0)})^2\}_{n\geq 0}$ and $\{(\mu_n^{(0)})^2\}_{n\geq 0}$, respectively.

By the procedure described in \cite{MA} one can find $q_1(x)$ and $h_1$ from $\varphi_1(1/2,\lambda)$ and $\varphi'_1(1/2,\lambda)$.
Let us describe this procedure in more detail.

Consider the corresponding prolonged Sturm-Liouville problem on the half-line:
\begin{equation*}
  \left\{\begin{array}{l}
  -y''+q(x)y=\lambda^2y, \quad x\in (0,\infty),\\
  y'(0)-h_1y(0)=0
  \end{array}\right.
\end{equation*}
with
\begin{equation*}
  q(x):=
  \begin{cases}
  q_1(x), \qquad &x\in[0,1/2]\\
  0,      \qquad &x\in[1/2,+\infty).
  \end{cases}
\end{equation*}
The so-called Jost solution $f(x,\lambda)$ of the above problem,
approaching $e^{i\lambda x}[1+o(1)]$ as $x\rightarrow+\infty$, $\lambda\in\mathbb{C}^{+}$,
can be represented in terms of the transformation operator:
\begin{equation}\label{s9}
f(x,\lambda)=e^{i\lambda x}+\int_x^{+\infty}K(x,t)e^{i\lambda t}dt,
\end{equation}
where $K(x,t)$ is the kernel function.
Since the potential vanishes for $x\geq1/2$, so the Jost solution equals $e^{i\lambda x} $. Then
\begin{equation*}
  f\left(1/2,\lambda\right)=e^{i\frac{\lambda}{2}}, \quad f'\left(1/2,\lambda\right)=i\lambda e^{i\frac{\lambda}{2}}.
\end{equation*}

Denote by $\psi_1(x,\lambda)$ the solution of Eq.(\ref{2.1}) under the initial conditions $\psi_1(0,\lambda)=0=\psi'_1(0,\lambda)-1$.
Then $\varphi_1(x,\lambda)$ and $\psi_1(x,\lambda)$ are two independent solutions of Eq.(2.1). Thus, when $0\leq x\leq 1/2$,
we have
\begin{equation*}
  f(x,\lambda)=c_1\varphi_1(x,\lambda)+c_2\psi_1(x,\lambda),
\end{equation*}
where $c_1$ and $c_2$ are constants.
In particular,
\begin{equation*}
    c_1\varphi_1\left(1/2,\lambda\right)+c_2\psi_1\left(1/2,\lambda\right)
    =f\left(1/2,\lambda\right)=e^{i\frac{\lambda}{2}},
\end{equation*}
\begin{equation*}
    c_1\varphi'_1\left(1/2,\lambda\right)+c_2\psi'_1\left(1/2,\lambda\right)
    =f'\left(1/2,\lambda\right)=i\lambda e^{i\frac{\lambda}{2}}.
\end{equation*}
Since $\varphi_1(1/2,\lambda)\psi'_1(1/2,\lambda)-\varphi'_1(1/2,\lambda)\psi_1(1/2,\lambda)\equiv 1$,
we get
\begin{equation*}
  c_1=e^{i\frac{\lambda}{2}}\left(\psi'_1\left(1/2,\lambda\right)-i\lambda\psi_1\left(1/2,\lambda\right)\right),
\end{equation*}
\begin{equation*}
  c_2=e^{i\frac{\lambda}{2}}\left(i\lambda\varphi_1\left(1/2,\lambda\right)-\varphi'_1\left(1/2,\lambda\right)\right).
\end{equation*}
The Jost function, corresponding to the Jost solution $f(x, \lambda)$, is defined as follows:
\begin{equation}\label{s3}
  f(\lambda):=f'(0,\lambda)-h_1f(0,\lambda)=e^{i\frac{\lambda}{2}}\left(i\lambda\varphi_1
  \left(1/2,\lambda\right)-\varphi'_1\left(1/2,\lambda\right)\right).
\end{equation}
Since the Jost function $f(\lambda)$ is known, we can construct the $S-$function of the problem on the semi-axis (see \cite{MA}):
\begin{equation}\label{s4}
  S(\lambda)=-\frac{f(-\lambda)}{f(\lambda)},
\end{equation}
and also the following function
\begin{equation}\label{s5}
  F_{S}(x)=\frac{1}{2\pi}\int_{-\infty}^{\infty}[S(\lambda)-1]e^{i\lambda x}d{\lambda}.
\end{equation}
Further we solve the Marchenko equation with respect to $K(x, t)$:
\begin{equation}\label{s6}
  K(x,t)+F_{S}(x+t)+\int_{x}^{\infty}K(x,s)F_S(s+t)ds=0,\ t\geq x,
\end{equation}
and construct the potential by the formula
\begin{equation}\label{s7}
  q_1(x)=-2\frac{dK(x,x)}{dx},
\end{equation}
where $q_1(x)$ is a real-valued function in $L^{2}(0,1/2)$, and
\begin{equation}\label{s8}
h_1=\frac{f'(0,-\lambda)+f'(0,\lambda)S(\lambda)}{f(0,-\lambda)+f(0,\lambda)S(\lambda)}.
\end{equation}

The constructed potential $q_1(x)$ generates the Robin-Dirichlet problem
with the characteristic function $\varphi_1(1/2,\lambda)$ and the Robin-Neumann problem with the
characteristic function $\varphi'_1(1/2,\lambda)$. Substituting these functions into (\ref{2.7}),
we obtain the initial characteristic function $\Phi(\lambda)$ with the set of zeros $\{\pm \lambda_n\}_{n\ge 0}$.
Consequently, the numbers $\{ \lambda_n^2 \}_{n \ge 0}$ are the eigenvalues
of the problem (\ref{2.1})-(\ref{2.4}), generated by $(a_1, a_2, h_1, h_2, q_1, q_2)$.
The proof is finished.
\end{proof}

The proof of Theorem~\ref{1} leads to the following algorithm for solving the inverse problem.

\medskip

\emph{Algorithm 1.} Let the data $S:=\left\{a_1, h_2, q_2(x), \{\lambda_n^2\}_{n\ge 0}\right\}$,
satisfying the conditions (1)-(3) of Theorem~\ref{1}, be given. We have to costruct $q_1(x)$, $h_1$ and $a_2$.\vspace{0.2cm}

\textbf{Step 1.} Construct $\Phi(\lambda)$, using $a_1$ and $\{\lambda_n^2\}_{n\ge 0}$, via (\ref{s1}).\vspace{0.2cm}

\textbf{Step 2.} Find $[q_1]+h_1$ and $a_2$, using the given $\{ \lambda_n^2 \}_{n \ge 0}$, $h_2$, $q_2$ and $a_1$, via
\eqref{defga}-\eqref{defa2}.

\textbf{Step 3.} Find $\varphi_2(1/2,\lambda)$ with its zeros $\{\nu_n^{(1)}\}_{n\in \mathbb{Z}^0}$ and
$\varphi_2'(1/2,\lambda)$ with its zeros $\{\mu_n^{(1)}\}_{n\in \mathbb{Z}^0}$, by using the given $h_2$ and $q_2(x)$,
via (\ref{2.5}) and (\ref{2.6}).\vspace{0.2cm}

\textbf{Step 4.} Construct $\psi_1^{(1)}(\lambda)$ via (\ref{2.12}) and (\ref{2.17}).\vspace{0.2cm}

\textbf{Step 5.} Find $\varphi_1(1/2,\lambda)$ by substituting $\psi_1^{(1)}(\lambda)$ into (\ref{2.5}).\vspace{0.2cm}

\textbf{Step 6.} Construct $\psi_2^{(1)}(\lambda)$ via (\ref{2.19}), (\ref{ew}) and (\ref{2.24}).\vspace{0.2cm}

\textbf{Step 7.} Find $\varphi_1'(1/2,\lambda)$ by substituting $\psi_2^{(1)}(\lambda)$ into (\ref{2.6}).\vspace{0.2cm}

\textbf{Step 8.} Find the Jost function $f(\lambda)$ from (\ref{s3}).\vspace{0.2cm}

\textbf{Step 9.} Construct $S(\lambda)$, $F_S(\lambda)$, and $K(x,t)$ via (\ref{s4})-(\ref{s6}), respectively.\vspace{0.2cm}

\textbf{Step 10.} Construct $q_1(x)$ from $K(x,t)$, via (\ref{s7}).\vspace{0.2cm}

\textbf{Step 11.} Find $f(0,\lambda)$ and $f'(0,\lambda)$ from $K(x,t)$ via (\ref{s9}).\vspace{0.2cm}

\textbf{Step 12.} Finally, construct $h_1$ via (\ref{s8}).

\medskip

\begin{remark} Note that $\dfrac{\varphi_1'(1/2, \lambda)}{\varphi_1(1/2, \lambda)}$ is the Weyl function, associated with the Sturm-Liouville problem on the interval $(0, 1/2)$ with the potential $q_1$. Therefore, one can also apply the method of spectral mappings to recover the potential from the Weyl function (see \cite{FY}). Another way to solve the half inverse problem with the discontinuity in the middle point of the interval considered is to adapt the methods from \cite{BO1,BO2}
for solving partial inverse problems on star-shaped graphs.
\end{remark}

\section{The case $0<d<1/2$}

This section deals with the existence result, corresponding to Theorem B, and the reconstructing algorithm for solution.
We follow the scheme of Section~2. First, the properties of the eigenvalues and of the characteristic function are studied,
then the functions $\varphi_1(1/2, \lambda)$ and $\varphi_1'(1/2,\lambda)$ are recovered, and, finally, these functions
are used for reconstruction of the
potential on the interval $(0, 1/2)$. However, there are some important differences. First, we have to recover also the parameters
$d$, $a_1$ and $a_2$ of the discontinuity. Second, at the final step we obtain the problem with the discontinuity on the interval $(0, 1/2)$,
and use the method of spectral mappings \cite{FY} to solve this problem.

\subsection{Direct problems}

In this subsection, the properties of the eigenvalues are provided, and several important relations for the characteristic
function are derived (see the formulas~\eqref{311},~\eqref{312} and~\eqref{313} below). We also show, how one can recover the characteristic function,
the discontinuity position $d$ and the jump $a_1$ from the eigenvalues.

In the case $0<d<1/2$, we consider (\ref{SL})-(\ref{jcs}) as a boundary value problem on a star graph, and rewrite it as follows:
\begin{eqnarray}
&& -y_j''+q_j(x)y_j=\lambda^2 y_j, \quad x\in \left(0, 1/2\right)\setminus\{d\}, \quad j=1, 2, \label{31}\\
&& y_j'(0)-h_j y_j(0)=0, \label{32}\\
&& y_1(d+)=a_1 y_1(d-), \label{33}\\
&& y_1'(d+)=a_1^{-1} y_1'(d-)+a_2 y_1(d-), \label{34}\\
&& y_1\left(1/2\right)=y_2\left(1/2\right), \label{35}\\
&& y_1'\left(1/2\right)+y_2'\left(1/2\right)=0. \label{36}
\end{eqnarray}
Here $q_j\in L^2(0, 1/2)$ for $j=1, 2,$ and $q_1 (x):=q(x)|_{[0, 1/2] \backslash\{d\}}$, $q_2(x):=q(1-x)|_{[0, 1/2]}$.

Denote by $\varphi_1 (x, \lambda)$ the solution of Eq.\eqref{31} for $j=1$, satisfying the initial conditions
$\varphi_1 (0, \lambda)-1=\varphi_1'(0, \lambda)-h_1=0$ and the continuity conditions \eqref{33} and \eqref{34}.
Then, according to \cite{YU}, one has
\begin{eqnarray}
&& \varphi_1 \left(1/2, \lambda\right)=b_1\cos\frac{\lambda}{2}+b_2\cos\lambda\left(\frac{1}{2}-2d\right)+f_1\left(1/2\right)\frac{\sin\frac{\lambda}{2}}{\lambda} \nonumber\\
&& \quad \quad \quad \quad \qquad +f_2\left(1/2\right)\frac{\sin\lambda\left(2d-1/2\right)}{\lambda}+\frac{\psi_1(\lambda)}{\lambda}, \label{37}
\end{eqnarray}
and
\begin{eqnarray}
&& \varphi_1'\left(1/2, \lambda\right)=\lambda \left(-b_1\sin\frac{\lambda}{2}+b_2\sin\lambda\left(2d-1/2\right)\right)+f_1\left(1/2\right)\cos\frac{\lambda}{2} \nonumber\\
&& \quad \quad \quad \quad \qquad -f_2\left(1/2\right)\cos\lambda\left(2d-1/2\right)+\psi_2(\lambda), \label{38}
\end{eqnarray}
where $\psi_j\in \mathcal{L}^{1/2}$ ($j=1,2$), and
\begin{eqnarray*}
&& f_1\left(1/2\right)=b_1\left(h_1+[q_1]\right)+\frac{a_2}{2},\\
&& f_2\left(1/2\right)=b_2\left(h_1-[q_1]+\int_0^{d}q_1(x)dx\right)-\frac{a_2}{2},\\
&& b_1=\frac{a_1+a_1^{-1}}{2}, b_2=\frac{a_1-a_1^{-1}}{2}.
\end{eqnarray*}
Denote by $\varphi_2 (x, \lambda)$ the solution of Eq.\eqref{31} for $j=2$, satisfying the initial conditions
$\varphi_2 (0, \lambda)=1$ and $\varphi_2'(0, \lambda)=h_2$. Then, according to \cite{FY,MA}, one has
\begin{eqnarray}
&& \varphi_2 \left(1/2, \lambda\right)=\cos\frac{\lambda}{2}+\left(h_2+[q_2]\right)\frac{\sin\frac{\lambda}{2}}{\lambda}+\frac{\psi_3(\lambda)}{\lambda}, \label{39}
\end{eqnarray}
and
\begin{eqnarray}
&& \varphi_2'\left(1/2, \lambda\right)=-\lambda\sin\frac{\lambda}{2}+\left(h_2+[q_2]\right)\cos\frac{\lambda}{2}+\psi_4(\lambda), \label{310}
\end{eqnarray}
where $\psi_j\in \mathcal{L}^{1/2}$ $(j=3,4)$.

The solution of the problem \eqref{31}-\eqref{34} possesses the form
$y_j=c_j\varphi_j(x,\lambda)$
$(j=1, 2)$, where $c_j$ are constants. Taking \eqref{35} and \eqref{36} into account, we represent the characteristic function of the problem \eqref{31}-\eqref{36} in the form
\begin{eqnarray}
&& \Phi(\lambda)=\varphi_1\left(1/2, \lambda\right)\varphi_2'\left(1/2, \lambda\right)+\varphi_1'\left(1/2, \lambda\right)\varphi_2\left(1/2, \lambda\right).
\label{311}
\end{eqnarray}
Using \eqref{37}-\eqref{310}, we derive
\begin{eqnarray}
\Phi(\lambda)\!=\!-\lambda(b_1\sin\lambda\!+\!b_2\sin\lambda(1-2d))\!+\!\omega_1\cos\lambda\!+\!\omega_2\cos\lambda(1-2d)\!+\!\psi(\lambda),
\label{312}
\end{eqnarray}
where $\psi\in \mathcal{L}^1$ and
\begin{eqnarray*}
&& \omega_1=b_1\left(h_1+h_2+[q_1+q_2]\right)+\frac{a_2}{2},\\
&& \omega_2=b_2\left(h_2-h_1+[q_1+q_2]-\int_0^{d}q_1(x)dx\right)+\frac{a_2}{2}.
\end{eqnarray*}

The eigenvalues $\{ \lambda_n^2 \}_{n \ge 0}$ of the boundary value problem \eqref{31}-\eqref{36} are real and simple.
They are supposed to be numbered in such a way, that the following asymptotic relation holds (see \cite{YU}):
$$
\lambda_n=\lambda_n^{0}+O\left(\frac{1}{\lambda_n^{0}}\right), \quad n \to \infty,
$$
where $\lambda_n^{0}$ are the zeros of the function $\Phi^{0}(\lambda)=-\lambda(b_1\sin\lambda+b_2\sin\lambda(1-2d))$,
$\lambda_n^0 \to +\infty$ as $n \to \infty$.

By using $\{\lambda_n^2\}_{n\ge 0}$, we can construct
\begin{eqnarray}
&& \Phi(\lambda)=C\prod_{n=0}^{\infty}\left(1-\frac{\lambda^2}{\lambda_n^2}\right),
\label{313}
\end{eqnarray}
where $C$ is a constant.

Introduce the function
$$
    \Psi(\lambda) = b_1^{-1} \Phi(\lambda) = C_0 \prod_{n = 0}^{\infty} \left( 1 - \frac{\lambda^2}{\lambda_n^2} \right), \quad C_0 = b_1^{-1} C.
$$
Taking (\ref{312}) into account, we find the constant
\begin{eqnarray}\label{001}
    C_0 = \frac{1}{2} \lim_{\tau \to +\infty} \tau e^{\tau} \left( \prod_{n = 0}^{\infty} \biggl( 1 + \frac{\tau^2}{\lambda_n^2} \biggr)\right)^{-1}.
\end{eqnarray}	
The relation (\ref{312}) implies
$$
    \lambda^{-1} \Psi(\lambda) + \sin \lambda = -\frac{b_2}{b_1} \sin \lambda (1 - 2 d) + O\left( \lambda^{-1} \right), \quad \lambda \in \mathbb R, \quad \lambda \to +\infty.
$$
Using the latter formula, we determine $a_1$ and $d$. Clearly,
\begin{eqnarray}\label{314}
    \frac{b_2}{b_1} = -\overline{\lim_{\lambda \to +\infty}} (\lambda^{-1} \Psi(\lambda) + \sin \lambda), \quad a_1 = \sqrt{\frac{1 + (b_2/b_1)}{1 - (b_2/b_1)}}.
\end{eqnarray}
Having $a_1$, we find $b_1$, $b_2$ and $\Phi(\lambda) = b_1 \Psi(\lambda)$.
According to (\ref{312}) and Rouche's theorem, the analytic function $(\lambda^{-1} \Psi(\lambda) + \sin \lambda)$ has a subsequence of zeros $\{ \theta_n \}_{n \ge n_0}$ with the following asymptotic behavior:
$$
   \theta_n = \frac{\pi n}{1 - 2 d} + O\left(n^{-1}\right), \quad n \to +\infty.
$$
Consequently,
\begin{eqnarray}\label{314d}
    d = \lim_{n \to \infty} \left( \frac{1}{2} - \frac{\pi n}{2 \theta_n}\right).
\end{eqnarray}

Introduce the function
\begin{eqnarray*}
&& h_2(\lambda):=\Phi(\lambda)+\lambda(b_1\sin\lambda+b_2\sin\lambda(1-2d))\\
&& \quad \quad \ \ =\omega_1\cos\lambda+\omega_2\cos\lambda(1-2d)+\psi(\lambda),
\end{eqnarray*}
then
\begin{eqnarray}
&& \omega_1=2\lim_{\tau \to +\infty}e^{-\tau}h_2(i\tau),
\label{315}\\
&& \omega_2=\lim_{\substack{m \to +\infty\\   m\in\mathbb{N}}}  (h_2(\lambda)-\omega_1\cos\lambda)|_{\lambda=\frac{2m\pi}{1-2d}}. \label{316}
\end{eqnarray}

\subsection{Reconstruction of $\varphi_1(1/2, \lambda)$ and $\varphi_1'(1/2, \lambda)$}

In this subsection, we focus on the following question.
Given the eigenvalues $\{\lambda_n^2\}_{n\ge 0}$, the potential $q_2$ and the coefficient $h_2$ of the right boundary condition,
how to find $a_1$, $a_2$, $d$, $q_1$ and $h_1$?
Section~3.1 presents algorithms for recovering $a_1$, $d$, $\omega_1$ and $\omega_2$ from $\{\lambda_n^2\}_{n\ge 0}$
(see \eqref{314}, \eqref{314d}, \eqref{315} and \eqref{316}). It remains to find $a_2$, $q_1$ and $h_1$. In order to do this,
we first reconstruct the characteristic functions $\varphi_1(1/2, \lambda)$ and $\varphi_1'(1/2, \lambda)$,
associated with the left half-interval $(0, 1/2)$. Our method is based on interpolation of entire functions
(see the relations \eqref{s15} and \eqref{s17} below).

Note that the potential $q_2$ and the coefficient $h_2$ are given,
so we can find the function $\varphi_2(1/2, \lambda)$ and its zeros $\{\nu_n^{(1)}\}_{n\in \mathbb{Z}^0}$ as well as
the function $\varphi_2'(1/2, \lambda)$ and its zeros $\{\mu_n^{(1)}\}_{n\in \mathbb{Z}^0}$
(recall that $\mathbb Z^0 := \{\pm 0, \pm 1, \pm 2, \ldots \}$). According to \cite{YA}, the following asymptotic
formulas are valid for $n \ge 0$:
\begin{eqnarray}
&& \nu_n^{(1)}=(2n+1)\pi+\frac{2[q_2]+2h_2}{(2n+1)\pi}+\frac{\beta_n}{n},
\label{317} \\
&& \mu_n^{(1)}=2n\pi+\frac{[q_2]+h_2}{n\pi}+\frac{\beta_n}{n}.
\label{318}
\end{eqnarray}
We suppose that $\nu_{-n}^{(1)} = -\nu_n^{(1)}$, $\mu_{-n}^{(1)} = -\mu_n^{(1)}$.

Substituting $\lambda=\nu_n^{(1)}$ into Eq.\eqref{311}, we get
\begin{eqnarray}
&& \varphi_1\left(1/2, \nu_n^{(1)}\right)=\frac{\Phi\left(\nu_n^{(1)}\right)}{\varphi_2'\left(1/2, \nu_n^{(1)}\right)}.
\label{319}
\end{eqnarray}
Note that
\begin{equation}\label{s14}
\left\{\begin{array}{l}
f_1\left(1/2\right)=\omega_1-b_1\left(h_2+[q_2]\right),\\
f_2\left(1/2\right)=-\omega_2+b_2\left(h_2+[q_2]\right),
\end{array}
\right.
\end{equation}
which are already known. To determine $\varphi_1(1/2, \lambda)$ from \eqref{37}, it is sufficient to recover $\psi_1(\lambda)$.
By selecting $\{\nu_n^{(1)}\}_{n\in \mathbb{Z}^0}$ as the nodes of interpolation, we will find the function $\psi_1(\lambda)$.
In order to do this, we first find
the values of the function $\psi_1(\lambda)$ at the nodes. From \eqref{37}, \eqref{38}, \eqref{317} and \eqref{319} we have
\begin{eqnarray}
&& \psi_1\left(\nu_n^{(1)}\right)=\nu_n^{(1)}\left[\frac{\Phi\left(\nu_n^{(1)}\right)}{\varphi_2'\left(1/2, \nu_n^{(1)}\right)}-b_1\cos\frac{\nu_n^{(1)}}{2}-b_2\cos\nu_n^{(1)}\left(1/2-2d\right)\right. \nonumber  \\
&& \quad \quad \quad \quad \quad \left.-f_1\left(1/2\right)\frac{\sin\frac{\nu_n^{(1)}}{2}}{\nu_n^{(1)}}-f_2\left(1/2\right)\frac{\sin\nu_n^{(1)}\left(2d-1/2\right)}{\nu_n^{(1)}}\right]. \label{320}
\end{eqnarray}

Consider the indices $n \ge 0$. In order to estimate $\psi_1\left(\nu_n^{(1)}\right)$, we need the following relations:
\begin{eqnarray*}
&& \Phi\left(\nu_n^{(1)}\right)=-2b_1\left([q_2]+h_2\right)+\nu_n^{(1)}b_2\sin\nu_n^{(1)}(1-2d)-\omega_1\cos\nu_n^{(1)} \\
&& \quad \quad \quad \quad \quad \ -\omega_2\cos\nu_n^{(1)}(1-2d)+\beta_n, \\
&& \varphi_2'\left(1/2, \nu_n^{(1)}\right)=(-1)^n(2n+1)\pi+\beta_n.
\end{eqnarray*}
Consequently,
\begin{eqnarray}
&& \frac{\Phi\left(\nu_n^{(1)}\right)}{\varphi_2'\left(1/2, \nu_n^{(1)}\right)}=\frac{1}{(-1)^{n-1}(2n+1)\pi}\left[2b_1\left([q_2]+h_2\right)-b_2\nu_n^{(1)}\sin\nu_n^{(1)}(1\!\!-\!\!2d)\!
\right. \nonumber  \\
&& \quad \quad \quad \quad \quad \quad \quad \ \ \quad +\omega_1\cos\nu_n^{(1)}\!+\!\omega_2\cos\nu_n^{(1)}(1\!\!-\!\!2d)\!+\!\beta_n]. \label{321}
\end{eqnarray}
In addition,
\begin{eqnarray}
&& \cos\nu_n^{(1)}=-1+\beta_n, \label{322} \\
&& \cos\frac{\nu_n^{(1)}}{2}=(-1)^{n-1}\left[\frac{[q_2]+h_2}{(2n+1)\pi}+\frac{\beta_n}{n}\right], \label{323}\\
&& \sin\frac{\nu_n^{(1)}}{2}=(-1)^{n}+\beta_n.
\label{324}
\end{eqnarray}
Substituting \eqref{321}-\eqref{324} into \eqref{320}, we get
\begin{eqnarray*}
&&  \psi_1\left(\nu_n^{(1)}\right)=(-1)^{n}\left[2b_1\left([q_2]+h_2\right)-\nu_n^{(1)}b_2\sin\nu_n^{(1)}(1-2d)\right. \\
&& \quad \quad \quad \quad \quad -\omega_1+\omega_2\cos\nu_n^{(1)}(1-2d)]+(-1)^{n}b_1\left([q_2]+h_2\right)\\
&& \quad \quad \quad \quad \quad -\nu_n^{(1)}b_2\cos\nu_n^{(1)}(1-2d)+(-1)^{n-1}f_1\left(1/2\right)\\
&& \quad \quad \quad \quad \quad -f_2\left(1/2\right)\sin\nu_n^{(1)}\left(2d-1/2\right)+\beta_n,
\end{eqnarray*}
which is equivalent to
\begin{eqnarray}
&&  \psi_1\left(\nu_n^{(1)}\right)=(-1)^{n-1}\left[2b_1\left([q_2]+h_2\right)-\omega_1-b_1\left([q_2]+h_2\right)\right.
\nonumber  \\
&& \quad \quad \quad \quad\quad \left.+f_1\left(1/2\right)\right] +b_2\nu_n^{(1)}\left[(-1)^{n}\sin\nu_n^{(1)}(1-2d)-\cos\nu_n^{(1)}\left(1/2-2d\right)\right]\nonumber  \\
&& \quad \quad \quad \qquad +(-1)^{n-1}\omega_2\cos\nu_n^{(1)}(1-2d)\nonumber\\
&& \quad \quad \quad \qquad-f_2\left(1/2\right)\sin\nu_n^{(1)}\left(2d-1/2\right)+\beta_n.
\label{325}
\end{eqnarray}
Note that the term in the first parentheses in \eqref{325} vanishes, thus
\begin{eqnarray}
&&  \psi_1\left(\nu_n^{(1)}\right)=b_2\nu_n^{(1)}\left[(-1)^{n}\sin\nu_n^{(1)}(1-2d)-\cos\nu_n^{(1)}\left(1/2-2d\right)\right]\nonumber  \\
&& \quad \quad \quad \qquad +(-1)^{n-1}\omega_2\cos\nu_n^{(1)}\left(\!1\!\!-\!\!2d\right)\nonumber\\
&& \quad \quad \quad \qquad -f_2\left(\!1/2\right)\sin\nu_n^{(1)}\left(2d\!\!-\!\!1/2\!\right)+\beta_n. \label{326}
\end{eqnarray}
By trigonometric calculation and \eqref{317}, we have
\begin{eqnarray*}
&& \sin\nu_n^{(1)}(1-2d)=(-1)^{n}\cos\nu_n^{(1)}\left(1/2-2d\right)+(-1)^{n-1}\frac{[q_2]+h_2}{(2n+1)\pi} \\
&& \quad \quad \quad \quad \quad \quad \quad \quad
\times\sin\nu_n^{(1)}\left(1/2-2d\right)+\frac{\beta_n}{n},\\
&& \cos\nu_n^{(1)}(1-2d)=(-1)^{n-1}\sin\nu_n^{(1)}\left(1/2-2d\right)+\beta_n.
\end{eqnarray*}
Substituting the latter formulas into \eqref{326}, we derive
\begin{eqnarray*}
&& \psi_1\left(\nu_n^{(1)}\right)=\left[\cos\nu_n^{(1)}\left(1/2-2d\right)-\frac{[q_2]+h_2}{(2n+1)\pi}\sin\nu_n^{(1)}\left(1/2-2d\right)\right.\\
&& \quad \quad \quad \quad \quad -\cos\nu_n^{(1)}\left(1/2-2d\right)\bigg]b_2\nu_n^{(1)}+\omega_2\sin\nu_n^{(1)}\left(1/2-2d\right)\\
&& \quad \quad \quad \quad \quad +f_2\left(1/2\right)\sin\nu_n^{(1)}\left(1/2-2d\right)+\beta_n\\
&& \quad \quad \quad \quad =-b_2\left([q_2]+h_2\right)\sin\nu_n^{(1)}\left(1/2-2d\right)+\left(\omega_2+f_2\left(1/2\right)\right)\\
&& \quad \quad \quad \quad \quad \times\sin\nu_n^{(1)}\left(1/2-2d)\right)+\beta_n\\
&& \quad \quad \quad \quad =\left[-b_2\left([q_2]+h_2\right)+\omega_2+f_2\left(1/2\right)\right]\sin\nu_n^{(1)}\left(1/2-2d\right)\\
&& \quad \quad \quad \quad \quad +\beta_n\\
&& \quad \quad \quad \quad =\beta_n, \quad n \ge 0.
\end{eqnarray*}
It is easy to check that $\psi_1(\nu_{-n}^{(1)}) = -\psi_1(\nu_n^{(1)})$,
so $\{ \psi_1(\nu_n^{1}) \}_{n \in \mathbb {Z}^0} \in l^2$.
Taking into account, that the function $\varphi_2(1/2, \lambda)$ is of sine type, we use interpolation (see \cite{LL}):
\begin{eqnarray}\label{s15}
\psi_1(\lambda)=\varphi_2\left(1/2, \lambda\right)\sum_{n\in \mathbb{Z}^0}\frac{\psi_1\left(\nu_n^{(1)}\right)}{\frac{d\varphi_2(1/2, \lambda)}{d\lambda}|_{\lambda=\nu_n^{(1)}}\left(\lambda-\nu_n^{(1)}\right)}.
\end{eqnarray}
Substituting $\psi_1(\lambda)$ into \eqref{37}, we obtain $\varphi_1(1/2, \lambda)$.

Similarly, we recover the function $\varphi_1'(1/2, \lambda)$. Replacing $\lambda$ by $\mu_n^{(1)}$ in Eq.\eqref{311}, we get
\begin{eqnarray}
\varphi_1'\left(1/2, \mu_n^{(1)}\right)=\frac{\Phi\left(\mu_n^{(1)}\right)}{\varphi_2\left(1/2, \mu_n^{(1)}\right)}. \label{327}
\end{eqnarray}
Using \eqref{37}, \eqref{38}, \eqref{318} and \eqref{327}, we obtain
\begin{eqnarray}\label{s16}
&& \psi_2\left(\mu_n^{(1)}\right)=\varphi_1'\left(1/2, \mu_n^{(1)}\right)-\mu_n^{(1)}\left(-b_1\sin\frac{\mu_n^{(1)}}{2}+b_2\sin\mu_n^{(1)}\left(2d-1/2\right)\right)\nonumber\\
&& \quad \quad \quad \quad \quad -f_1\left(1/2\right)\cos\frac{\mu_n^{(1)}}{2}+f_2\left(1/2\right)\cos\mu_n^{(1)}\left(2d-1/2\right)\nonumber\\
&& \quad \quad \quad \quad =\frac{\Phi\left(\mu_n^{(1)}\right)}{\varphi_2\left(1/2, \mu_n^{(1)}\right)}-\mu_n^{(1)}\left(-b_1\sin\frac{\mu_n^{(1)}}{2}+b_2\sin\mu_n^{(1)}\left(2d-1/2\right)\right)\nonumber\\
&& \quad \quad \quad \quad \quad -f_1\left(1/2\right)\cos\frac{\mu_n^{(1)}}{2}+f_2\left(1/2\right)\cos\mu_n^{(1)}\left(2d-1/2\right).
\end{eqnarray}
Similarly to the estimation of $\{\psi_1\left(\nu_n^{(1)}\right)\}_{n\in \mathbb{Z}^0}$,
we also get $\{ \psi_2\left(\mu_n^{(1)}\right) \}_{n \in \mathbb Z^0}\in l^2$. Taking into account, that the function
$g_2(\lambda):=\frac{\lambda\varphi_2'(1/2, \lambda)}{\lambda^2-(\mu_0^{(1)})^2}$ is of sine type, we use interpolation (see \cite{LL}):
\begin{eqnarray}\label{s17}
\psi_2(\lambda)=g_2(\lambda)\sum_{0\neq n=-\infty}^{+\infty}\frac{\psi_2\left(\mu_n^{(1)}\right)}{\frac{dg_2(\lambda)}{d\lambda}|_{\lambda=\mu_n^{(1)}}\left(\lambda-\mu_n^{(1)}\right)}
+g_2(\lambda)\frac{\psi_2(0)}{g_2'(0)\lambda}.
\end{eqnarray}
Substituting $\psi_2(\lambda)$ into \eqref{38}, we can obtain $\varphi_1'(1/2, \lambda)$.
Here
\begin{eqnarray}\label{yd}
\psi_2(0)=\frac{\Phi(0)-\varphi_1(1/2,0)\varphi_2'(1/2,0)}{\varphi_2(1/2,0)}-f_1(1/2)+f_2(1/2).
\end{eqnarray}
If $\varphi_2(1/2,0)=0$, then we have to shift the spectrum by a constant.

Now we consider the following boundary value problems:
\begin{eqnarray}
&& -y_1''+q_1(x)y_1=\lambda^2 y_1, \quad x\in \left(0, 1/2\right)\setminus\{d\}, \label{328}\\
&& y_1'(0)-h_1 y_1(0)=0, \label{329}\\
&& y_1(d+)=a_1 y_1(d-), \label{330}\\
&& y_1'(d+)=a_1^{-1} y_1'(d-)+a_2 y_1(d-), \label{331}\\
&& y_1\left(1/2\right)=0, \label{332} \\
\rm or \nonumber \\
&& y_1'\left(1/2\right)=0. \label{333}
\end{eqnarray}

It is easy to see that $\Delta_1(\lambda):=\varphi_1(1/2, \lambda)$ is the characteristic function of the problem
\eqref{328}-\eqref{332}, and $\Delta_2(\lambda):=\varphi_1'(1/2, \lambda)$ is the characteristic function of
the problem \eqref{328}-\eqref{331}, \eqref{333}. Certainly, we can get the zeros $\{\mu_n^{(1)}\}_{n\in \mathbb{Z}^0}$ of the
function $\Delta_2(\lambda)$. Define the norming constants
$$\alpha_n:=\int_0^{1/2}\varphi_1^2(x, \mu_n^{(1)})dx, \quad n\geq 0.$$
Then (see \cite{FY})
\begin{eqnarray} \label{s18}
\alpha_n=-\dot{\Delta}_2(\mu_n^{(1)})\Delta_1(\mu_n^{(1)}),
\end{eqnarray}
where $\dot{\Delta}_2(\lambda)=\frac{d}{d\lambda}\Delta_2(\lambda)$.
The data $\{(\mu_n^{(1)})^2, \alpha_n\}_{n\geqslant0}$ are called the spectral data of the boundary value problem \eqref{328}-\eqref{331},
\eqref{333}. Together with \eqref{328}-\eqref{331}, \eqref{333}, we consider another problem of the same form but with different
coefficients $\tilde{q}_1(x)$, $\tilde{h}_1$ and $\tilde{a}_2$. If a certain symbol $\gamma$ denotes an object related
to the previous problem, then $\tilde{\gamma}$ will denote the analogous object related to the new problem.
Let two boundary value problems be such that
\begin{eqnarray} \label{334}
&& \sum_{n=0}^{\infty}\xi_n|\mu_n^{(1)}|<\infty,
\end{eqnarray}
where $\xi_n:=|\mu_n^{(1)}-\tilde{\mu}_n^{(1)}|+|\alpha_n-\tilde{\alpha}_n|$.

\subsection{Inverse problems: existence and algorithm}

In this subsection, we first prove the following existence theorem, corresponding to the uniqueness Theorem~B.

\begin{theorem}\label{2}
Given data $S:=\{h_2, q_2(x), \{\lambda_n^2\}_{n\ge 0}\}$, satisfying the following conditions:

\begin{enumerate}
\item $\lambda_0^2<\lambda_1^2<\cdots<\lambda_n^2<\cdots$;\\

\item the series
$$
   \prod_{n = 0}^{\infty} \left( 1 - \frac{\lambda^2}{\lambda_n^2} \right)
$$
converges uniformly on compact sets of $\mathbb C$, and the function
$$
   \Psi(\lambda) = C_0 \prod_{n = 0}^{\infty} \left( 1 - \frac{\lambda^2}{\lambda_n^2} \right),
$$
where $C_0$ is defined by \eqref{001}, admits the representation
$$
    \Psi(\lambda) = -\lambda (\sin \lambda + \theta \sin \lambda (1 - 2d)) + \kappa_1 \cos \lambda + \kappa_2 \cos\lambda(1 - 2 d) + \psi_0(\lambda),
$$
where $\theta$, $\kappa_1$, $\kappa_2$ and $d$ are real constants, $-1< \theta < 1$, $|\theta| + |\kappa_2| > 0$, $0 < d < 1/2$, $\psi_0 \in \mathcal L^1$.
\end{enumerate}

Suppose that the reals $\{(\mu_n^{(1)})^2, \alpha_n\}_{n\geqslant0}$ constructed above
have the following properties: $\alpha_n>0$, $(\mu_n^{(1)})^2\neq(\mu_m^{(1)})^2$ $(n\neq m)$,
and there exist data $\{\tilde{q}_1(x), \tilde{a}_2, \tilde{h}_1\}$, such that \eqref{334} holds.

Then there exist a real-valued function $q_1(\cdot)\in L^2(0,1/2)$ and real numbers $a_1$, $a_2$, $d$ and $h_1$,
such that the spectrum of the problem \eqref{31}-\eqref{36}, generated by $(q_1, q_2)$ and $(a_1, a_2, d, h_1, h_2)$,
coincides with $\{\lambda_n^2\}_{n\geqslant0}$.
\end{theorem}

\begin{proof}

Let data $S = \{ h_2, q_2(x), \{ \lambda_n^2 \} \}_{n \ge 0}$ satisfy the conditions of the theorem. Using these data, we construct the parameters $a_1$ and $d$,
as it was described in~Section~3.1. Then we also construct the functions $\varphi_1(1/2, \lambda)$, $\varphi_1'(1/2, \lambda)$ and, finally,
the data $\{ (\mu_n^{(1)})^2, \alpha_n \}_{n \ge 0}$, as it was described in Section~3.2.

Recall that, for the reals $\{(\mu_n^{(1)})^2, \alpha_n\}_{n\geqslant0}$ to be the spectral data
for a certain boundary value problem \eqref{328}-\eqref{331} and \eqref{333}, the following conditions are necessary and sufficient:
$\alpha_n>0$, $(\mu_n^{(1)})^2\neq(\mu_m^{(1)})^2$ $(n\neq m)$, and there exist data $\{\tilde{q}_1(x), \tilde{a}_2, \tilde{h}_1\}$,
such that \eqref{334} holds (see \cite{FY}). By the assumption of the theorem, those conditions are valid.
Hence we can apply the method of spectral mappings,
described in Section~4.4 (Thereom~4.4.5 and Lemma~4.4.2 in \cite{FY}) to obtain $a_2$, $h_1$ and $q_1$:
\begin{eqnarray}\label{s19}
&& q_1(x)=\tilde{q}_1(x)+\varepsilon(x),\nonumber~~~~~~~~~~~~~~~~~~~~~~~~~~~~~~~~~~~~~~~~~~~~~~~~~~~~~~~~~~~~~~~~~~~~~~~\\
&& a_2=\tilde{a}_2+(a_1^{-1}-a_1^3)\varepsilon_0(d-0),\\
&& h_1=\tilde{h}_1-\varepsilon_0(0).\nonumber
\end{eqnarray}
Here we use the triple $\{\tilde{q}_1$, $\tilde{a}_2$, $\tilde{h}_1\}$, described in the theorem statement and satisfying~\eqref{314}, together with $\{a_1,d\}$, constructed via
(\ref{314}) and (\ref{314d}). The functions $\varepsilon(x)$, $\varepsilon_0(x)$ are defined by (4.4.69) in \cite{FY}.

Thus, the data $\{(\mu_n^{(1)})^2, \alpha_n\}_{n\geqslant0}$ are the spectral data of the problem \eqref{328}-\eqref{331}, \eqref{333} with
the values $\{ q_1, h_1, d, a_1, a_2 \}$, constructed above. Consequently, the constructed functions $\varphi_1(1/2, \lambda)$ and $\varphi_1'(1/2, \lambda)$
are the characteristic functions of the eigenvalue problems \eqref{328}-\eqref{332} and \eqref{328}-\eqref{331}, \eqref{333}, respectively.
Hence the function $\Phi(\lambda)$, satisfying~\eqref{311}, is the characteristic function of the problem \eqref{31}-\eqref{36}.
By construction, $\Phi(\lambda)$ has the zeros $\{ \pm \lambda_n \}_{n \ge 0}$, so we have proved that $\{ \lambda_n^2 \}_{n \ge 0}$
are the zeros of the constructed problem~\eqref{31}-\eqref{36}.

\end{proof}

Finally, we arrive at the following algorithm for solving the inverse problem under consideration.

\medskip

\emph{Algorithm 2.} Let the data $S:=\left\{h_2, q_2(x), \{\lambda_n^2\}_{n\ge 0}\right\}$,
satisfying the conditions of Theorem \ref{2}, be given. We have to construct $a_1, a_2, d, h_1$ and $q_1(x)$.\vspace{0.2cm}

\textbf{Step 1.} Construct $a_1$ and $d$, using $\{\lambda_n^2\}_{n\ge 0}$, via (\ref{314}) and (\ref{314d}).\vspace{0.2cm}

\textbf{Step 2.} Construct $\Phi(\lambda)$ from $\{\lambda_n^2\}_{n\ge 0}$ with $C=b_1C_0$ and $b_1=\frac{a_1+a_1^{-1}}{2}$, via (\ref{313}) and (\ref{001}).\vspace{0.2cm}

\textbf{Step 3.} Find $\omega_1$ and $\omega_2$ appearing in (\ref{312}), via (\ref{315}) and (\ref{316}).\vspace{0.2cm}

\textbf{Step 4.} Find $\varphi_2(1/2,\lambda)$ with its zeros $\{\nu_n^{(1)}\}_{n\in \mathbb{Z}^0}$ and
$\varphi_2'(1/2,\lambda)$ with its zeros $\{\mu_n^{(1)}\}_{n\in \mathbb{Z}^0}$ from the given $h_2$ and $q_2(x)$,
via (\ref{39}) and (\ref{310}).\vspace{0.2cm}

\textbf{Step 5.} Find $f_1(1/2)$ and $f_2(1/2)$ appearing in (\ref{37}) and (\ref{38}), via (\ref{s14}).\vspace{0.2cm}

\textbf{Step 6.} Construct $\psi_1(\lambda)$ via (\ref{320}) and (\ref{s15}).\vspace{0.2cm}

\textbf{Step 7.} Find $\varphi_1(1/2,\lambda)$ by substituting $\psi_1(\lambda)$ into (\ref{37}).\vspace{0.2cm}

\textbf{Step 8.} Construct $\psi_2(\lambda)$ via (\ref{s16}), (\ref{yd}) and (\ref{s17}).\vspace{0.2cm}

\textbf{Step 9.} Find $\varphi_1'(1/2,\lambda)$ by substituting $\psi_2(\lambda)$ into (\ref{38}).\vspace{0.2cm}

\textbf{Step 10.} Find the norming constants $\{\alpha_n\}_{n\geq 0}$ via (\ref{s18}).\vspace{0.2cm}

\textbf{Step 11.} Together with $\{a_1,d\}$, constructed by (\ref{314}) and (\ref{314d}), take $\{\tilde{a}_2, \tilde{h}_1, \tilde{q}_1\}$, satisfying (\ref{334}).\vspace{0.2cm}

\textbf{Step 12.} Find $\varepsilon(x)$ and $\varepsilon_0(x)$ by (4.4.69) in \cite{FY}.\vspace{0.2cm}

\textbf{Step 13.} Finally, construct $q_1(x)$, $a_2$ and $h_1$ via (\ref{s19}).

\bigskip

\noindent {\bf Acknowledgments.} The authors would like to thank the referees for valuable comments.
The research work was supported in part by the National Natural Science Foundation of China
(11871031 and 11611530682) and
the author Bondarenko
was supported by Grant 1.1660.2017/4.6 of the Russian Ministry of Education
and Science, and by Grant 19-01-00102 of the Russian
Foundation for Basic Research.

\end{document}